\documentclass[11pt,a4paper,twoside]{article}
\usepackage[T1]{fontenc}
\usepackage[english]{babel}
\usepackage{amssymb, amsmath, amsfonts, amsthm}
\usepackage{enumerate}
\usepackage{colonequals}
\usepackage{empheq,fancybox}
\usepackage{wasysym}
\usepackage[small]{titlesec}
\usepackage[noblocks]{authblk}

\usepackage{microtype}
\usepackage{ellipsis}
\usepackage[BCOR=20mm,DIV=11]{typearea}
 \newcommand{\hm}[1]{\leavevmode{\marginpar{\tiny%
 $ \hbox to 0mm{\hspace*{-0.5mm} $ \leftarrow $ \hss}%
 \vcenter{\vrule depth 0.1mm height 0.1mm width \the\marginparwidth}%
 \hbox to
 0mm{\hss $ \rightarrow $ \hspace*{-0.5mm}} $ \\\relax\raggedright #1}}}

\newcommand{\euler}{\mathrm{e}}
\newcommand{\drm}{\mathrm{d}}
\newcommand{\dvol}{\mathrm{dvol}}
\newcommand{\RR}{\mathbb{R}}
\newcommand{\CC}{\mathbb{C}}
\newcommand{\NN}{\mathbb{N}}

\renewcommand{\epsilon}{\varepsilon}

\DeclareMathOperator{\diam}{\mathop{diam}}

\DeclareMathOperator{\Vol}{\mathop{Vol}}

\newtheorem{theorem}{Theorem}[section]
\newtheorem{lemma}[theorem]{Lemma}
\newtheorem{proposition}[theorem]{Proposition}
\newtheorem{corollary}[theorem]{Corollary}
\theoremstyle{definition}

\theoremstyle{remark}
\newtheorem{remark}[theorem]{Remark}

\usepackage{color}
	\definecolor{darkred}{rgb}{0.5,0,0}
	\definecolor{darkgreen}{rgb}{0,0.5,0}
	\definecolor{darkblue}{rgb}{0,0,0.5}

\begin{document}
\author{Christian Rose}
\affil{Technische Universit\"at Chemnitz, Fakult\"at f\"ur Mathematik, Germany}
\title{Heat kernel upper bound on Riemannian manifolds with locally uniform Ricci curvature integral bounds}
\maketitle
\begin{abstract}
This article shows that under locally uniformly integral bounds of the negative part of Ricci curvature the heat kernel admits a Gaussian upper bound for small times. This provides general assumptions on the geometry of a manifold such that certain function spaces are in the Kato class. Additionally, the results imply bounds on the first Betti number.
\end{abstract}
\section{Introduction}
One of the most important invariants of a Riemannian manifold is given by its heat kernel. Because there is no explicit representation of it in general, an important and interesting topic in geometric analysis is the small-time behavior of this function in terms of the underlying geometry. Especially Gaussian upper bounds are of particular interest. The article \cite{Grigoryan-99} provides plenty examples for Riemannian manifolds possessing a heat kernel with Gaussian behavior. A very prominent one is given by the following.
\begin{theorem}[\cite{Grigoryan-99}]
Let $M$ be a Riemannian manifold with bounded geometry. Then the heat kernel $p_t(\cdot,\cdot)$ can be bounded by
\begin{align}\label{upperbound}
\sup_{x\in M} p_t(x,x)\leq Ct^{-n/2}, \quad 0<t\leq 1\text,
\end{align}
where $C>0$ depends on the geometry of $M$.
\end{theorem}
Another interesting example would be the case of complete, connected and non-compact Riemannian manifolds with non-negative Ricci curvature. In this case the heat kernel satisfies a bound like \eqref{upperbound} for all $t>0$. The techniques used in \cite{Grigoryan-99} for the proof of such bounds are the volume doubling property and the so-called relative Faber-Krahn inequalities. For another approach see \cite{SaloffCoste-02}.\\
Considering the above theorem leads to the question under which weaker geometric conditions a bound like \eqref{upperbound} is satisfied.  
Of course, the assumption that the Ricci curvature is bounded from below is a quite strong condition. One could think of a manifold with Ricci curvature staying around zero almost everywhere but providing a set where the Ricci curvature has a deep well. In this case, the bound \eqref{upperbound} gets worse as it should be, because it grows exponentially as the Ricci curvature bound decreases. From an analytic point of view it therefore seems natural to replace the boundedness by an integral condition. A first step was done by Gallot. For a map $f\colon M \to \RR$ let $v_-:= \sup(0,-v)$ and $v_+:=\sup(0,v)$. Denote by $\rho\colon M\to\RR$ the function whose values are the smallest eigenvalues of the Ricci tensor.
\begin{theorem}[\cite{Gallot-88}]\label{Gallot}
Let $D>0$, $\lambda>0$ and $p>n\geq 3$. There are explicitly computable constants $K_G, C_G>0$ such that for any compact Riemannian manifold $M$ with $\dim M=n$, $\diam(M)\leq D$ and Ricci curvature satisfying
\begin{align}\label{curvaturecondition}
\frac 1{\Vol(M)}\int \left(\frac{\rho_-}{n-1}-\lambda^2\right)_+^{p/2}\dvol\leq K_G,
\end{align}
the heat kernel $p_t(x,y)$ can be bounded from above by
\[
\sup_{x,y\in M}p_t(x,y)\leq C_Gt^{-p/2},\quad 0<t\leq 1\text.
\]
\end{theorem}
There are several things which should be noted here. On one hand, the heat kernel does not provide the right local dimension $n$, which depends on the technique. Gallot obtains a global isoperimetric inequality of dimension $p$ controlling the negative part of the Ricci curvature. This in turn yields the above estimate for the heat kernel by a symmetrization procedure. On the other hand, he assumes that $M$ is compact and the Ricci curvature is globally $L^p$-bounded. Yang noticed in \cite{Yang-92} that it is possible to get local isoperimetric inequalities assuming a local $L^p$-bound on the negative part of the Ricci curvature. 
\begin{theorem}[\cite{Yang-92}]\label{Yang}
Let $M$ be a complete connected Riemannian manifold of dimension $n$, $p>n/2$, $r>0$ and $v>0$. Assume that for all $x\in M$, $\Vol(B(x,r))\geq v$. There are explicit constants $K_Y, C_Y>0$ such that if 
$$\int_{B(x,3r)}\rho_-^p<K_Y\text,$$
then we have for all $\Omega\subset B(x,r)$ with smooth boundary
$$\frac{A(\partial\Omega)}{\Vol(\Omega)^\frac{n-1}{n}}\geq C_Y.$$
\end{theorem}
Unfortunately, this theorem provides a result for the heat kernel only in the non-collapsing case, where one assumes that the volumes of balls cannot approach zero. \\
The purpose of this article is to show that on complete connected manifolds, either compact or not, one can generalize Yang's result to the collapsing case under the assumption that one has locally uniform $L^p$-bounds on the negative part of Ricci curvature. A precise definition will be given in the next section. The main observation which is necessary for the proof was made in \cite{PetersenWei-00}. Assuming a locally uniform $L^p$-bound on the Ricci curvature yields a local volume doubling condition for balls. This is sufficient to generalize Theorem\ref{Yang} to the collapsing case, which is one of our main results. The result implies  a local version of Lemma 7.16 in \cite{Grigoryan-99}, which provides lower bounds for fractions of volumes of balls. We then improve Theorem \ref{Yang} to the collapsing case with explicit dependence on all parameters. The local uniformity of the curvature yields local isoperimetric inequalities in every ball of the same radius. A covering technique and a result of \cite{Grigoryan-09} then imply the heat kernel bound with the right local dimension in a quantitative way. In particular, in the compact case we cover Theorem \ref{Gallot} and obtain the local dimension $n$.\\
Upper bounds of the heat kernel for small times allow to characterize certain function spaces.  In \cite{GueneysuPost-13} the authors show that under certain conditions of the heat kernel, $L^p$-spaces are in the so called Kato-class. Let us recall the statement briefly. Let $M$ be a Riemannian manifold of dimension $n$ and denote by $(P_t)_{t\geq 0}$ the heat semigroup on $M$ and $p_t(\cdot,\cdot)$ its minimal heat kernel. For a non-negative measurable function $V\colon M\to\RR_+$ and $\beta>0$ let
$$b_{\rm{Kato}}(V,\beta):=\sup_{n\in \NN}\int_0^\beta\Vert P_t(V\wedge n)\Vert_\infty \drm t\text.$$
We say that a function $V\colon M\to \CC$ is in the Kato class $\mathcal{K}(M)$ if 
$$\lim_{\beta\to 0}  b_{\rm{Kato}}(\vert V\vert,\beta)=0\text.$$ 
\begin{theorem}[\cite{GueneysuPost-13}]
Let $M$ be a Riemannian manifold of dimension $n\geq 3$ and $p>n/2$. Assume that there are $t_0, C>0$ such that
\begin{align}
\sup_{x\in M} p_t(x,x)\leq Ct^{-n/2}, \quad 0<t\leq t_0\text.
\end{align}
Then we have 
\begin{align}\label{LpsubsetKato}
L^p(M)+L^\infty(M)\subset \mathcal{K}(M)\text.
\end{align}
\end{theorem}
Note that the local dimension $n$ of the heat kernel does not play a significant role to obtain qualitative results like \eqref{LpsubsetKato} and could also be replaced by another constant, affecting the choice of $p$. In the recent paper \cite{RoseStollmann-15} Peter Stollmann and the author used the so-called extended Kato class to derive bounds on topological invariants for compact manifolds. A measurable function $V\colon M\to \CC$ is in the extended Kato class if for some $\beta>0$
$$b_{\rm{Kato}}(\vert V\vert,\beta)<1\text.$$
Using Theorem \ref{Gallot} it was shown that $b_{\rm{Kato}}(\rho_-,\beta)$ can be bounded in terms of the averaged $L^p$-norm of $\rho_-$, leading to bounds on the first Betti number via semigroup domination. The result shows how the upper bound scales with the involved quantities, but it is not explicit. The local results of the present article give explicit bounds on the first Betti number with a locally uniform $L^p$-bound on the negative part of the Ricci curvature.
\section{Preliminaries}
In this article $M$ is always a complete connected Riemannian manifold without boundary of dimension $n$. For $x\in M$ and $r>0$ 
let $B(x,r)$ be the open ball around $x$ with radius $r$. 
Denote by $\dvol$ the volume measure of $M$ and for $x\in M$ and $r>0$ let $V(x,r)$ the volume of $B(x,r)$. Moreover, $A(H)$ gives the $(n-1)$-dimensional measure of a hypersurface $H\subset M$. 
For any $x\in M$ and $r>0$ we define the local isoperimetric constant for $B(x,r)$ to be the quantity
$$I(x,r):=\inf \frac{A(\partial \Omega)}{\Vol(\Omega)^{\frac{n-1}{n}}}\text,$$
where $\Omega$ runs over all relatively compact subset of $B(x,r)$ with smooth boundary. \\
Let $\rho\colon M\to \RR$ the smallest 
eigenvalue function of the Ricci tensor and write $\rho_-=\sup(0,-\rho)$. We define for $p>n/2$ the quantities
\begin{align*}
\kappa(p,x_0,s)&:=\int_{B(x_0,s)}\rho_-^{p}\dvol,\\[1.5ex]
\bar\kappa(p,x_0,s)&:=\frac {\kappa(p,x_0,s)}{V(x_0,s)},\\[1.5ex]
\bar\kappa(p,s)&:=\sup_{x_0\in M}\left(\bar\kappa(p,x_0,s)\right)^{1/p}.
\end{align*}
Note that it is convenient to work with the scale invariant quantity $R^2\bar\kappa(p,R)$, as explained in 
\cite{PetersenWei-00}.
If there is a lower Ricci curvature bound, $R^2\bar\kappa(p,R)$ is small for small $R$.
For convenience, we recall an important fact about the quantity $R^2\bar\kappa(p,R)$, and quote the following statement.
\begin{proposition}[\cite{PetersenWei-00}, Theorem 2.1]\label{prop:PetersenWei}
Assuming that $p>n/2$, there are explicit constants $\varepsilon=\varepsilon(n,p)>0$ and $D=D(p)>0$ such that for any $R>0$ with 
$$R^2\bar\kappa(p,R)< \varepsilon$$ 
we have for all $x\in M$ and $0<r< s\leq R$
$$D\frac{V(x,s)}{V(x,r)}\leq \left(\frac sr\right)^n\text.$$
In particular we can set $\varepsilon=\frac{(2p-n)^{p+1}}{2n(n-1)(2p-1)}$ and $D=\left(\frac{2-\sqrt 2}{2}\right)^{2p}$.
\end{proposition}
\begin{remark}
The locally uniform bound of the Ricci curvature and the local volume doubling property above allow to compare $r^2\bar\kappa(p,r)$ for different values $r$. Assume that 
the assumptions of Proposition \ref{prop:PetersenWei} are satisfied. Then it is a direct consequence that for all $0<r_1<r_2\leq R$ 
\begin{align}\label{scaling1}
r_1^2\bar\kappa(p,r_1)&\leq D^{-1}\left(\frac{r_2}{r_1}\right)^{\frac np -2}r_2^2\bar\kappa(p,r_2)\text.
\end{align}
The other direction looks similar and can be seen by a packing argument. Assume that $0<r_1<r_2$ and that $r_1^2\bar\kappa(p,r_1)<\varepsilon$. Then
\begin{align}\label{scaling2}
r_2^2\bar\kappa(p,r_2)\leq \left(\frac{2^n}{D}\right)^\frac{1}{p}\left(\frac{r_2}{r_1}\right)^2r_1^2\bar\kappa(p,r_1)\text.
\end{align}
\end{remark}
\section{Local isoperimetric inequalities}

Assuming that the Ricci curvature integral bounds are uniformly small provides local isoperimetric inequalities. 
\begin{theorem}\label{prop:isoperimetric}
Let $2p>n\geq 3$ and $r>0$ and assume that $r<\diam(M)$. 
There are explicit constants $K_L=K_L(n,p,r)>0$ and $C_I=C_I(n,p,r)>0$ such that if 
\begin{equation}
r^2\bar\kappa(p,r)< K_L\text,
\end{equation}
for all $x_0\in M$ we have
\begin{equation}\label{eq:isoinequ}
I(x_0,r)\geq C_I\, \left(\frac{V(x_0,r)}{(1+r^{2p})r^n}\right)^{1+1/n}\text.
\end{equation}
\end{theorem}
In order to prove Theorem \ref{prop:isoperimetric} we use Yang's idea from \cite{Yang-92}. To overcome the issue of non-collapsing we use a local volume doubling result which is based on ideas from \cite{Grigoryan-99}. Bounding the ratio of the volumes of two intersecting balls in terms of their radii, the ratio will also be bounded from below.
\begin{lemma}\label{volumecomparison}
Let $\varepsilon>0$ be as in Proposition \ref{prop:PetersenWei} and $R>0$. Assume that $M$ is non-compact or $R\leq\diam(M)/3$. If
\[
R^2\bar\kappa(p,R)<\varepsilon\text,
\]
then there are explicit constants $a=a(n,p)>0$, $b=b(n,p)>0$ and $\eta=\eta(n,p)>0$ such that for all $0<r\leq s\leq R/3$ and balls $B(x,s)$ and $B(y,r)$ such that $\overline{B(x,r)}\cap\overline{B(y,s)}\neq \emptyset$
\begin{align}
a\left(\frac sr\right)^\eta\leq \frac{V(x,s)}{V(y,r)}\leq b\left(\frac sr \right)^n.
\end{align}
\end{lemma}
\begin{proof}{(following \cite{Grigoryan-99})}
We know that $d(x,y)\leq r+s\leq 2R/3$ and therefore Proposition \ref{prop:PetersenWei} implies
\begin{align*}
\frac{V(x,s)}{V(y,r)}\leq \frac{V(y,2s+r)}{V(y,r)}\leq \frac{V(y,3s)}{V(y,r)}\leq b\left(\frac sr\right)^n
\end{align*}
where
\[
b:=3^nD^{-2p}\text.
\]
Assume now $s=3r$ and $x=y$, then there is a $\xi\in M$ with $d(x,\xi)=2r$, $B(\xi,r)\cap B(x,r)=\emptyset$, such that, using the upper inequality for $V(\xi,r)$
\[
V(x,s)=V(x,3r)\geq V(x,r)+V(\xi,r)\geq\left(1+b^{-1}\right)V(x,r)\text.
\]
For the general case we know that there is a $k\in\NN$ such that
$$3^k\leq \frac sr<3^{k+1},$$
so we can conclude
\begin{align*}
V(x,s)\geq V(x,3^kr)&\geq V(x,r)\left(1+b^{-1}\right)^k\geq b^{-1}(1+b^{-1})^kV(y,r)\\[1.5ex]
&=b^{-1}(1+b^{-1})^{-1}(1+b^{-1})^{k+1}V(y,r)\\
&\geq aV(y,r)\left(\frac sr\right)^{\eta}
\end{align*}
where
\begin{align*}
a:=b^{-1}(1+b^{-1})^{-1}=(b+1)^{-1}\quad\text{and}\quad \eta :=\log_3(1+b^{-1})\text.
\end{align*}
\end{proof}
Since we want to follow the main steps of \cite{Yang-92} it seems convenient to recall some definitions and results from this text.
Let $x_0\in M$ and denote by $S_{x_0}\subset T_{x_0}M$ the unit sphere in the tangent space at $x_0$. By $c_m$ we denote the volume of the unit disc in $\RR^m$. Given a 
subset $\hat S\subset S_{x_0}$ and $\rho>0$ let
\[
 \Gamma(\hat S,\rho):=\{y=\exp_{x_0} r\theta\mid 0\leq r<\rho, \theta\in \hat S, d(x,y)=r\}
\]
the geodesic cone with base point $x_0$ and length $r$.

\begin{proposition}[see \cite{Yang-92}]\label{prop:Yang}
Let $M$ be a Riemannian manifold, $x_0\in M$, $\hat S\subset S_{x_0}$, $p>n/2$, $\tau>0$. Set 
 $\delta=\frac{2p-n}{2p-1}$, $\hat \omega=\Vol(\hat S)$, 
\[
  k:=\int_{\Gamma(\hat S,\rho)}\rho_-^p,
  \quad\text{and}\quad r_0:=\frac{\tau^{1/\delta}}{1+\tau}(h(n,p)^{-1}\hat\omega k^{-1})^{1/(2p-n)}.
 \]
Then, for every $0\leq r\leq\rho$
\begin{align}\label{uppervolumecone}
 \Vol(\Gamma(\hat S,r))\leq
 \begin{cases}
 (1+\tau)^{n-1}\hat\omega n^{-1}r^n\colon & 0\leq r\leq r_0,\\
 c(n,p,\tau)k((1-\delta)r+\delta r_0)^{2p} \colon & r\geq r_0.
 \end{cases}
\end{align}
Here, $h(n,p):=\left(2-1/p\right)^p\left(\frac{p-1}{2p-n}\right)^{p-1}\!\!$ 
and $c(n,p,\tau):=(1+\tau^{-1})^{2p-1}\frac{2p-1}{2p(n-1)}h(n,p)$.
\end{proposition}
The proof of the Proposition above requires $k>0$ because it appears in the denominator of $r_0$. In the case $k=0$ the volume bound of the geodesic cone follows directly from the Bishop-Gromov inequality, \cite{Chavel-84}. Observe that in the case $\hat S_{x_0}=S_{x_0}$ it is $k=\kappa(p,x_0,r)$ for $\rho=r$.
We are now able to prove Proposition \ref{prop:isoperimetric}.
\begin{proof}[Proof of Theorem \ref{prop:isoperimetric}]
For the proof we distinguish between the non-compact and compact case. \\
\emph{The non-compact case:}
Define the constant 
$$G:=\left(\frac D{2^n}\right)^{1/p}(16(b+2)^{2/\eta})^{-1}\min\left\{\varepsilon, \left(\frac {D^{p+1}r^na}{c(n,p,\tau^{-1})}\right)^{\frac 1p}\left(4(b+2)^{1/\eta}\right)^{2}r^2\right\}\text,$$
with $\tau=\tau(n,p,r)>0$ to be chosen later.
Let $\Omega\subset B(x_0,r)$ have smooth boundary $\partial\Omega$. By $b>0$ from Lemma \ref{volumecomparison} we define
$$s:=4(b+2)^{1/\eta}r\text.$$
Given $x\in\Omega$, let 
$\hat S_x\subset S_x$ denote the set of unit tangent vectors $v$ such that the corresponding geodesic 
$\exp_x sv, s>0$, is a minimal geodesic joining $x$ to some point in $B(x_0,s/3)\setminus B(x_0,r)$ and choose 
$x\in\Omega$ such that $\hat S_x$ has minimal volume. Croke's inequality \cite{Croke-80,Chavel-93} tells us that
\[
 \frac{A(\partial\Omega)}{\Vol(\Omega)^{(n-1)/n}}\geq \frac{c_{n-1}}{(c_n/2)^{1-\frac 1n}}
 \left(\frac{\hat\omega}{c_{n-1}}\right)^{1+1/n}\text.
\]
Therefore it suffices to find a lower bound for $\hat \omega$. Since $r^2\bar\kappa(p,r)<G$, \eqref{scaling2} gives
\begin{align}\label{bigradius}
s^2\bar\kappa(p,s)<\min\left\{\varepsilon, D\left(\frac {Dar^n}{c(n,p,\tau^{-1})}\right)^{\frac 1p}s^2\right\}\text,
\end{align}
such that Lemma \ref{volumecomparison} and the definition of $s$ yield
\begin{align}\label{eq:lowervolumecone}
\Vol(\Gamma(\hat S_{x_0},s/3+r))
&\geq V(x_0,s/3)-V(x_0,r)\nonumber \\
&\geq \left(a\left(\frac s{3r}\right)^\eta-1\right)V(x_0,r)\nonumber\\
&\geq a V(x_0,r)\text.
\end{align}
On the other hand, \eqref{bigradius} and \eqref{scaling1} imply
\begin{align}\label{smallradius}
\left(\frac{r+s/3}{s}\right)^{\frac np-2}(r+s/3)^2\bar\kappa(p,r+s/3)\leq \left(\frac {Dar^n}{c(n,p,\tau^{-1})}\right)^{\frac 1p}s^2
\end{align}
and therefore
\begin{align}\label{middleradius}
\kappa(p,x_0,r+s/3)\leq \frac {aD}{c(n,p,\tau^{-1})}V(x_0,r+s/3)r^n(r+s/3)^{-n}\text.
\end{align}
Apply Proposition \ref{prop:PetersenWei} to \eqref{eq:lowervolumecone} and use \eqref{middleradius}  to get
\begin{align}\label{conecurvature}
\Vol(\Gamma(\hat S_{x_0},s/3+r))&\geq aD\left(\frac r{s/3+r}\right)^n V(x_0,r+s/3)\nonumber\\
&\geq c(n,p,\tau^{-1})\kappa(p,x_0,r+s/3)\text.
\end{align}
In the following we want to choose $\tau$ such that in Proposition \ref{prop:Yang} we can always use the first case in \eqref{uppervolumecone} . If one assumes $s/3+r\geq r_0$, combining the second case of \eqref{uppervolumecone} and \eqref{conecurvature} yields
\[
  c(n,p,\tau^{-1})\leq c(n,p,\tau) (s/3+r)^{2p}
\]
or
\[
\tau^{2p-1}\leq (s/3+r)^{2p}\text.
\]
Therefore, defining $\tau:=s^\frac{2p}{2p-1}>0$ suffices. 
In this case one can conclude that $r+s/3\leq r_0$, and in turn
\[
 \Vol(\Gamma(\hat S_{x_0},s/3+r))\leq (1+\tau)^{n-1}\frac{\hat\omega}n(s/3+r)^n.
\]
Combining the upper and lower bound for the volume of the geodesic cone we get
\begin{align*}
\frac{\hat\omega}{c_{n-1}}&\geq \frac{na}{c_{n-1}}(1+\tau)^{1-n}(s/3+r)^{-n}V(x_0,r)\\
&=\frac{na}{c_{n-1}}(4(b+2)^{\frac{1}{\eta}}+1)^{-n}(1+s^\frac{2p}{2p-1})^{1-n}\frac{V(x_0,r)}{r^n}\text.
\end{align*}
By assumption it is $2p>n$, and therefore $$(1+s^\frac{2p}{2p-1})^{n-1}\leq 2^n (4(b+2)^{\frac{1}{\eta}}+1)^{2p}(1+r^{2p})\text.$$ This yields
\begin{align}\label{endup}
\frac{\hat\omega}{c_{n-1}}
&\geq C_N\frac{V(x_0,r)}{(1+r^{2p})r^n}\text,
\end{align}
where $$C_N:=\frac{na}{c_{n-1}}2^{-n}(4(b+2)^{\frac{1}{\eta}}+1)^{-2p-n}\text.$$
\emph{The compact case:} Let $G$ and $s$ be as above. If $r\leq \frac{\diam(M)}{4(b+2)^{1/\eta}}$, then $s\leq \diam (M)$ and the proof above applies. Otherwise, let 
$$\tilde r:= \frac{\diam(M)}{4(b+2)^{1/\eta}}\quad \text{such that}\quad \tilde s:= \diam(M), \quad \tilde\tau:=\tilde s^\frac{2p}{2p-1}$$
and  $$\tilde G:=\left(\frac D{2^n}\right)^{1/p}(16(b+2)^{2/\eta})^{-1}\min\left\{\varepsilon, \left(\frac {D^{p+1}\tilde r^na}{c(n,p,\tilde\tau^{-1})}\right)^{\frac 1p}\left(4(b+2)^{1/\eta}\right)^{2}\tilde r^2\right\}\text.$$
Assuming $r^2\bar\kappa(p,r)<\tilde G$, \eqref{scaling1} gives $\tilde r^2\bar \kappa(p,\tilde r)<\tilde G$. The proof above applies again and leads to \eqref{endup} with $r$ replaced by $\tilde r$. Since $r^2\bar\kappa(p,r)<\epsilon$, Proposition \ref{prop:PetersenWei} implies the inequality
\[
\frac{\hat\omega}{c_{n-1}}
\geq C_c\frac{V(x_0,r)}{(1+r^{2p})r^n}\text,
\]
where $C_c:=DC_N$. We end up by setting $$C_I:= \frac{c_{n-1}}{(c_n/2)^{1-\frac 1n}}C_c^{1+1/n}$$
and 
$$K_L:=\begin{cases} G\colon &\text{$M$ non-compact or $r\leq \frac{\diam(M)}{4(b+2)^{1/\eta}}$},\\
\tilde G\colon &\text{otherwise.}\end{cases}$$
\end{proof}

\section{Heat kernel bounds and the Kato class}
As mentioned in the introduction, it is sufficient to obtain a Gaussian upper bound for the heat kernel for small times if we have an isoperimetric inequality for every ball of the same radius. The reason is that such isoperimetric inequalities yield relative Faber-Krahn inequalities in these balls. Using a covering argument, Theorem 15.14 in \cite{Grigoryan-09} implies an upper bound for the heat kernel for small times in the collapsing case.
\begin{theorem}\label{theorem:upperbound}
Assume $2p>n\geq 3$, $r>0$ and assume that $r<\diam(M)$. Let $D=D(n,p)>0$ and $K_L=K_L(n,p,r)>0$ as above and 
\begin{equation}
r^2\bar\kappa(p,r)<DK_L\text.
\end{equation}
Then, there exists an explicit $C=C(n,p,r)>0$ such that for all $x\in M$ and $0<t\leq r^2/4$
\[
 p_t(x,x)\leq \frac C{V(x,r)^{n+1}}\,t^{-n/2}\text.
\]
\end{theorem}

\begin{proof}
Let $x\in M$. By Proposition \ref{prop:isoperimetric} it follows from our assumption on $r^2\bar\kappa(p,r)$ that 
$$I(x,r/2)\geq C_I\left(\frac{V(x,r/2)}{(4^p+2^{-2p}r^{2p})2^{-n}r^n}\right)^{1+\frac 1n}=C_I2^{\frac{(2p+n)(n+1)}{n}}\left(\frac{V(x,r/2)}{(4^{2p}+r^{2p})r^n}\right)^{1+\frac 1n}\text.$$
This isoperimetric bound in $B(x,r/2)$ is 
equivalent to the statement that for all $\Omega\in B(x,r/2)$ with smooth boundary $\partial\Omega$ we have
\[
 \frac{A(\partial\Omega)}{\Vol(\Omega)}\geq C_I2^{\frac{(2p+n)(n+1)}{n}}\left(\frac{V(x,r/2)}{(4^{2p}+r^{2p})r^n}\right)^{1+\frac 1n}\Vol(\Omega)^{-1/n}\text.
\]
By Cheeger's theorem, this implies that the first Dirichlet eigenvalue of $\Omega$, $\lambda_1(\Omega)$, is bounded from below,
\[
 \lambda_1(\Omega)\geq C_I^22^{\frac{2(2p+n)(n+1)}{n}-2}\left(\frac{V(x,r/2)}{(4^{2p}+r^{2p})r^n}\right)^{2+2/n}\Vol(\Omega)^{-2/n}\text.
\]
That means, for every ball $B(x,r/2)$ we have a relative Faber-Krahn inequality. Since $(B(x,r/2))_{x\in M}$
is a cover for $M$, Theorem 15.14 in \cite{Grigoryan-09} implies that, after collecting all the constants, for all
$x\in M$ and $t\in (0,r^2/4)$ the heat kernel $p_t$ on $M$ satisfies
\begin{align*}
p_t(x,x)\leq C_h C_I^{-n}2^{n-(2p+n)(n+1)}(4^p+r^{2p})^{2n+2}r^{2n(n+1)}V(x,r/2)^{-n-1}t^{-n/2}
\end{align*}
where
\[
 C_h:= \frac{2\sqrt{17}4^{\frac{n+1}{n+2}+n+2}\euler^2}{4^{n+2}-1} 
 \leq 10\euler^2\frac{4^{n+3}}{4^{n+2}-1}.
\]
Appling Lemma \ref{volumecomparison} to $V(x,r/2)$ proves the theorem.
\end{proof}
\begin{corollary}\label{cor:uniform}
Assume that for $M$ the assumptions of Theorem \ref{theorem:upperbound} hold and additionally
$$T:=\liminf_{x\in M}V(x,r)/r^k> 0$$
for some $k\in \NN$. 
Then we have for all $x\in M$ and $0<t\leq r^2/4$
\[
 p_t(x,x)\leq \frac C{T^{n+1}r^{k(n+1)}}\,t^{-n/2}.
\]
\end{corollary}
The corollary holds as well when assuming that the injectivity radius $i_M$ of $M$ is positive and $r<i_M/2$, see \cite{Croke-80}, but we need the lower bound only for one $r$ and not for a whole scale. The same argument works under the non-collapsing assumption of Theorem \ref{Yang}.

\begin{corollary}
Under the assumptions of Corollary \ref{cor:uniform} we have
 $$\displaystyle{L^p(M)+L^\infty(M)\subset \mathcal{K}(M)}\text.$$

\end{corollary}
In the special case that we assume that $M$ is compact Theorem \ref{theorem:upperbound} and Proposition \ref{prop:PetersenWei} enable to prove a version of Theorem \ref{Gallot} providing the local dimension $n$ of the heat kernel. Contrary to Gallot's result we do not use a global isoperimetric inequality.
\begin{corollary}
Assume that $M$ is compact and $2p>n\geq 3$. There is an explicit constant $C_G=C_G(n,p,\diam(M), \Vol(M))>0$ such that if
$$(\diam(M))^2\bar\kappa(p,\diam(M))<K_L,$$
then for all $x\in M$ and $0<t\leq \frac 13\diam(M)$
$$p_t(x,x)\leq C_G\, t^{-n/2}\text.$$
\end{corollary}
\begin{proof}
We have by \eqref{scaling1}
\[
\left(1/3\diam(M)\right)^2\bar\kappa(p,1/3\diam(M))\leq 3^{n/p-2}(\diam(M))^2\bar\kappa(p,\diam(M))^2<K_L\text.
\]
Theorem \ref{theorem:upperbound} and Proposition \ref{prop:PetersenWei} imply for all $x\in M$ and $t\in(0,1/3\diam(M)]$
\[
p_t(x,x)\leq \frac{C(n,p,1/3\diam(M))}{V(x,1/3\diam(M))}t^{-n/2}\leq C(n,p,1/3\diam(M))\!\left(\frac{3^n}{\Vol(M)}\right)^{n+1}\!\!t^{-n/2}\text.
\]
\end{proof}
\section{An upper bound for $b_1(M)$}
As explained in the introduction, Kato class techniques are a powerful tool to derive bounds on topological invariants. For compact manifolds $M$ it was shown in \cite{RoseStollmann-15} that it is possible to bound the first Betti number $b_1(M)$ in terms $\bar\kappa(p,\diam(M))$ using Theorem \ref{Gallot}. Because $b_1(M)=\dim H^1(M)$ can be bounded by the trace of the semigroup of the Hodge-Laplacian on $1$-forms, the semigroup domination principle yields a bound on $b_1(M)$ using Theorem \cite{Gallot-88}. The bound shows how an upper bound on $b_1(M)$ scales with the occuring quantities, but is not explicit. This depends on the symmetrization technique used by Gallot to derive the local dimension $p>n$ of the heat kernel. Since Theorem \ref{theorem:upperbound} gives the right local dimension $n$ and the constants are explicitly computed, it is possible to change the occuring quantities in the proof of Corollary 5.7 in \cite{RoseStollmann-15}. Using the same techniques leads to a generalization of this result under locally uniform bounds on the negative part of the Ricci curvature.
\begin{corollary}\label{Betti}
Let $M$ a compact manifold of dimension $n\geq 3$, $p>n/2$ and $r<\diam(M)$. There are explicit constants $K_B=K_B(n,p,\diam(M),\Vol(M))>0$ and $B=B(n,p,\diam(M), r^2\kappa(p,r))>0$ such that if $$r^2\bar\kappa(p,r)<K_B\text,$$
then $$b_1(M)\leq B\text.$$
\end{corollary}
\begin{proof}
Set
$$K_B:=\min\left\{\left(\frac{r}{\diam(M)}\right)^2\left(\frac D{2^n}\right)^{1/p} K_L,\, T,\, \Vol(M)^{n}\right\}$$
where $$T:= \left(\frac D{2^n}\right)^{1/p}\frac{2p-n}{2p}\frac{4^\frac{2p-n}{2p}\Vol(M)^{n+1-\frac1p}}{C(n,p,r)}\left(\frac{r}{\diam(M)}\right)^{n(n+1)}\text.$$
First one has to get a global bound of the heat kernel for small times. Since $K_B\leq K_L$, Theorem \ref{theorem:upperbound} gives 
$$p_t(x,x)\leq \frac {C(n,p,r)}{V(x,r)^{n+1}}t^{-n/2},\quad t\in(0,r^2/4)\text,$$
%where 
%$$C=C(n,p,r)=C_hC_I^{-n}2^{n-(2p+n)(n+1)}(4^p+r^{2p})^{2n+2}r^{2n(n+1)}\text.$$
Furthermore, \eqref{scaling2} implies $(\diam(M))^2\bar\kappa(p,\diam(M))<\epsilon$, such that we have a global volume doubling property on $M$. Setting $$C_1:=\frac {C(n,p,r)}{\Vol(M)^{n+1}}\left(\frac{\diam(M)}{r}\right)^{n(n+1)}$$ 
gives
$$p_t(x,x)\leq C_1t^{-n/2},\quad t\in(0,r^2/4)\text.$$ 
The following is an adaption to section 5 of \cite{RoseStollmann-15}. We only explain the changes which have to be done in this article to get the desired result.
If we denote by $\Vert\cdot\Vert_{p,q}$ the operator norm from $L^p(M)$ to $L^q(M)$, we get for the semigroup generated by the Laplace-Beltrami operator $\Delta\geq 0$ on $M$ for $p>n$
$$\Vert\euler^{t\Delta}\Vert_{p,\infty}\leq C_1^{1/p}t^{-n/2p},\quad 0<t\leq r^2/4\text.$$
Proposition 5.1 in \cite{RoseStollmann-15} tells us 
\[
b_{\rm{Kato}}(\rho_-,r^2/4)\leq C_2\bar\kappa(p,\diam(M))
\]
with $$ C_2:= \Vol(M)^{1/p}\frac{2p}{2p-n}\left(\frac{r^2}4\right)^\frac{2p-n}{2p}C_1\text.$$
To get $L^p$-$L^q$-smoothing for the Hodge-Laplacian as explained in \cite{RoseStollmann-15} one has to force $b_{\rm{Kato}}(\rho_-,r^2/4)$ to be smaller than one. Using \eqref{scaling2}, the definition of $K_B$ leads to
$$(\diam(M))^2\bar\kappa(p,\diam(M))\leq \left(\frac{2^n}{D}\right)^{1/p}\left(\frac{\diam(M)}{r}\right)^2r^2\bar\kappa(p,r)<C_2^{-1}\text.$$
Define 
$$C_3:= D^{1/p}2^{-n/p}r^{-2}\text.$$
Corollary 5.7 in \cite{RoseStollmann-15} implies by the defnition of $K_B$
\begin{align*}\label{Betti1}
b^1(M)&\leq d\left(\frac 2{1-C_3r^2\bar\kappa(p,r)}\right)^{\left(1+\frac 4{r^2}\right)\frac 2{1-C_3r^2\bar\kappa(p,r)}}\frac {C(n,p,r)}{\Vol(M)^n}\left(\frac{\diam(M)}{r}\right)^{n(n+1)}\nonumber\\
&\leq d\left(\frac 2{1-C_3r^2\bar\kappa(p,r)}\right)^{\left(1+\frac 4{r^2}\right)\frac 2{1-C_3r^2\bar\kappa(p,r)}}\frac {C(n,p,r)}{r^2\bar\kappa(p,r)}\left(\frac{\diam(M)}{r}\right)^{n(n+1)}\text.
\end{align*}
\end{proof}

\begin{remark}
During the publication process we learned that independently Dai, Wei and Zhang also worked on local integral Ricci curvature bounds, see \cite{DaiWeiZhang-16}. They gain slightly better estimates on the Sobolev constant but using completely different and much more complicated techniques, and develop it for completely different purposes.
\end{remark}
\section*{Acknowledgement}
The author wants to thank Peter Petersen and Alexander Grigoryan for stimulating discussions about the topic and their hospitality during the authors research stays at UCLA and the University of Bielefeld.
%\bibliographystyle{alpha}
%\bibliography{ChristianRoseBIBO.bib}
\def\cprime{$'$}

\end{document}